\documentclass[12pt]{amsart}
\usepackage{fullpage,amssymb}
\newcommand{\Z}{\mathbb Z}
\newcommand{\SL}{{\text {\rm SL}}}
\newcommand{\leg}[2]{\genfrac{(}{)}{}{}{#1}{#2}}
\theoremstyle{plain}
\newtheorem{theorem}{Theorem}[section]
\newtheorem{corollary}[theorem]{Corollary}
\newtheorem{lemma}[theorem]{Lemma}
\newtheorem{proposition}[theorem]{Proposition}
\theoremstyle{definition}

\title{$p$-adic properties of coefficients of weakly holomorphic modular forms}
\author{Darrin Doud, Paul Jenkins}
\address{
Department of Mathematics, Brigham Young University, Provo, UT  84602}
\date{\today}
\email{doud@math.byu.edu, jenkins@math.byu.edu}
\begin{document}

\begin{abstract}
We examine the Fourier coefficients of modular forms in a canonical
basis for the spaces of weakly holomorphic modular forms of weights
$4$, $6$, $8$, $10$, and $14$, and show that these coefficients are
often highly divisible by the primes 2, 3, and 5.

\end{abstract}
\maketitle
\section{Introduction}

Let $\Delta(\tau)$ be the unique normalized cusp form of weight $12$
for $\SL_2(\Z)$ given by \[\Delta(\tau) = q \prod_{n=1}^\infty
(1-q^n)^{24} = \sum_{n=1}^{\infty} \tau(n) q^n.\]  Here, as is
traditional, $\tau$ is a complex number in the upper half plane, $q
= e^{2 \pi i \tau}$, and $\tau(n)$ is Ramanujan's tau-function.  Let
\[j(\tau) = q^{-1} + 744 + 196884q + \ldots = q^{-1} + 744 + \sum_{n
\geq 1} c(n) q^n\] be the standard modular $j$-function, a weakly
holomorphic modular form of weight $0$ for $\SL_2(\Z)$.

The arithmetic properties of the Fourier coefficients $\tau(n)$ of
$\Delta$ and $c(n)$ of $j$ have been studied extensively.  For
instance, Ramanujan \cite{Ramanujan} proved that
\[\tau(2n) \equiv 0 \pmod{2},\quad\tau(3n) \equiv 0 \pmod{3},\quad\tau(5n) \equiv 0 \pmod{5},\]
and Lehner \cite{Lehner1, Lehner2} proved that
\[c(2^a 3^b 5^c 7^d n) \equiv 0 \pmod{2^{3a+8} 3^{2b+3} 5^{c+1}
7^d}.\]  In fact, Lehner's results are much more general; he proved
similar congruences for the coefficients of many weakly holomorphic
modular functions on $\Gamma_0(p)$.  In this paper we will prove
such divisibility results for the Fourier coefficients of a large
class of weakly holomorphic modular forms of positive integer weight
on $\SL_2(\Z)$.

Recall that a \emph{weakly holomorphic modular form} of weight $k
\in 2\Z$ satisfies the modular equation \begin{equation*} \label{mod}
f\left(\frac{a\tau+b}{c\tau+d}\right) = (c\tau + d)^k
f(\tau)\end{equation*} for all $\left(\begin{smallmatrix}a&b\\c&d
\end{smallmatrix}\right)$ in some finite index subgroup $\Gamma$ of
$\SL_2(\Z)$ and is holomorphic on the upper half plane, but may have
poles at any or all of the cusps of $\Gamma$.  We denote the space
of holomorphic modular forms by $M_k(N)$ and denote by $M_{k}^!(N)$
the space of weakly holomorphic modular forms on $\Gamma_0(N) = \{
\left(\begin{smallmatrix}a&b\\c&d
\end{smallmatrix}\right) \in \SL_2(\Z) : c \equiv 0 \pmod{N} \}$.
When $N = 1$ we write simply $M_k$ and $M_k^!$, and note that any
form $f \in M_k^!$ is uniquely determined by its Fourier expansion
 $\sum_{n \geq n_0} a(n) q^n$,  where $n_0 \geq 0$ if $f
\in M_k$.

Because $\SL_2(\Z)$ has only one cusp (at $\infty$), there is a very
nice canonical basis for the space $M_{k}^!$, indexed by the order
of the pole at $\infty$.  To define this basis (as
in~\cite{Duke-Jenkins}) write $k = 12\ell +
k'$, where $\ell \in \Z$ and $k' \in \{0, 4, 6, 8, 10, 14\}$, so
that if $\ell \geq 0$, the space of cusp forms of
weight $k$ has dimension $\ell$. For every integer $m \geq -\ell$, there exists
a unique weakly holomorphic modular form $f_{k,m} \in M_{k}^!$ with
a $q$-expansion of the form
\[f_{k, m}(\tau) = q^{-m} + {O}(q^{\ell+1});\] together
these $f_{k, m}$ form a basis for $M_k^!$. It is straightforward to
see that any modular form $f = \sum a(n) q^n$ in $M_{k}^!$, if its
first few Fourier coefficients are known, may easily be written in
terms of these basis elements as \[f = \sum_{n_0 \leq n \leq \ell}
a(n) f_{k, -n}.\]

The basis elements $f_{k, m}$ may be directly constructed from
$\Delta$, $j$, and the Eisenstein series $E_{k'}$, where we let $E_0
= 1$. From the standard valence formula, we know that
\begin{equation}\label{valence} \textrm{ord}_\infty(f) \leq \ell\end{equation} for all $f \in
M_{k}^!$. It is also clear that the form $\Delta^\ell E_{k'}$ has
order $\ell$ at $\infty$, so it must be $f_{k, -\ell}$; it is unique
because the difference of any two such forms will have a Fourier
expansion that is ${O}(q^{\ell+1})$, which must be zero by
(\ref{valence}). We can then construct the $f_{k, m}$ iteratively by
multiplying $f_{k, m-1}$ by $j(\tau) \in M_0^!$ to get a form of
weight $k$ with Fourier expansion beginning $q^{-m}$, and then
subtracting appropriate integer multiples of $f_{k, i}$ (where
$-\ell < i < m$) to eliminate the $q^{-i}$ terms in the Fourier
expansion.  This construction shows that \[f_{k,m} = \Delta^\ell
E_{k'} F_{k, D}(j),\] where $F_{k, D}(x)$ is a monic polynomial in
$x$ of degree $D = \ell + m$ with integer coefficients.

We define the Fourier coefficients $a_k(m, n)$ of these basis
elements by \[f_{k,m}(\tau) = q^{-m} + \sum_n a_k(m, n) q^n,\]
noting that $a_k(m, n) = 0$ when $m < -\ell$ or $n \leq \ell$ or
when $m$ or $n$ are not integers.  Since each of $E_{k'}, \Delta,
j$, and $F_{k,m}$ has integer coefficients, it follows that
$a_{k}(m, n) \in \Z$.

These basis elements are studied extensively in~\cite{Duke-Jenkins};
for instance, for basis elements which are not cusp forms, the zeros
in the fundamental domain are all shown to lie on the unit circle.
Additionally, the following theorem is proved.
\begin{theorem}[\cite{Duke-Jenkins}, Theorem 2] \label{genfn}
For any even integer $k$, the basis elements $f_{k, m}$ satisfy the
generating function
\[\sum_{m \geq -\ell} f_{k,m}(z) q^m = \frac{f_{k, -\ell}(z) f_{2-k,
1+\ell}(\tau)}{j(\tau) - j(z)}.\] \end{theorem}  Replacing $k$ with
$2-k$ and switching $\tau$ and $z$, we find the following beautiful
duality of Fourier coefficients.
\begin{corollary}[\cite{Duke-Jenkins}, Corollary 1] \label{duality}
For any even integer $k$ and any integers $m, n$, the equality
\[a_k(m, n) = -a_{2-k}(n, m)\] holds for the Fourier coefficients of
the modular forms $f_{k, m}$ and $f_{2-k, n}$.
\end{corollary}

We note that these bases for $M_{k}^!$ closely parallel those of
half integral weight defined by Zagier in his work on traces of
singular moduli~\cite{Zagier}.  For the weights $1/2$ and $3/2$,
Zagier constructed modular forms on $\Gamma_0(4)$ satisfying a plus
space condition and with Fourier expansions \[\begin{array}{rl} f_d
=
q^{-d} + {O}(q) \in M_{1/2}^!  & (0 \leq d \equiv 0, 1 \pmod{4}), \\
g_D = q^{-D} + {O}(1) \in M_{3/2}^!  & (0 \leq D \equiv 0, 3
\pmod{4}).\end{array}\]  These modular forms have a generating
function and satisfy a duality theorem very similar to
Theorem~\ref{genfn} and Corollary~\ref{duality}, and bases with
similar properties for all half integral weights are constructed
in~\cite{Duke-Jenkins2}.

The Fourier coefficients of Zagier's $f_d$ and $g_D$ can be
interpreted as traces and twisted traces of singular moduli, and
have been widely studied. For instance, Ahlgren and
Ono~\cite{Ahlgren-Ono} proved many congruences for these traces (and
their associated half integral weight Fourier coefficients) modulo
$p^s$, and gave an elementary argument that if $p$ splits in
$\mathbb{Q}(\sqrt{-d})$, then $\textrm{Tr}(p^2 d)$ is congruent to
$0 \pmod{p}$. Edixhoven~\cite{Edixhoven} extended their observation,
proving that if $\leg{-d}{p} = 1$, then $\textrm{Tr}(p^{2n} d)
\equiv 0 \pmod{p^n}$.  An elementary proof of this result using
Hecke operators was given by the second author in~\cite{Jenkins},
and Boylan~\cite{Boylan} exactly computed $\textrm{Tr}(2^{2n} d)$,
obtaining even stronger congruences and divisibility results for $p
= 2$.

With these congruences for Fourier coefficients of forms of half
integral weight as a model, it is natural to ask whether similar
divisibility results exist for the Fourier coefficients $a_k(m, n)$
of the integral weight basis elements $f_{k, m}$. A result similar
to Edixhoven's appears in~\cite{Duke-Jenkins}, where it is proved
that for the weights $k = 4, 6, 8, 10$, and $14$, if $p^r | n$ and
$p \nmid m$, it is true that $p^{(k-1)r} | a_k(m, n)$. Thus, if $(m,
n) = 1$, we have $n^{k-1} | a_k(m, n)$.  This theorem seems to be
sharp if $p > 7$. However, looking at the divisibility of these
Fourier coefficients by smaller primes, it appears that the $a_k(m,
n)$ are divisible by higher powers of these primes.  For instance,
certain coefficients of \[f_{4, 1}(\tau) = E_4(\tau)(j(\tau)-984) =
q^{-1} + \sum_{n \geq 1} a_4(1, n) q^n\] factor in the following
way.
\[\begin{array}{rcl}
a_4(1, 2)&=& 2^{10}\cdot 5 \cdot 13327 \\
a_4(1, 4)&=& 2^{13} \cdot 3^2 \cdot 13 \cdot 113 \cdot 2543 \\
a_4(1, 5)&=& 2^3 \cdot 5^4 \cdot 19 \cdot 9931 \cdot 7639 \\
a_4(1, 8)&=& 2^{16} \cdot 3 \cdot 5^2 \cdot 293 \cdot
15918317 \\
a_4(1, 10)&=& 2^{11} \cdot 3^6 \cdot 5^4 \cdot 2184176461 \\
a_4(1, 20)&=& 2^{14} \cdot 5^4 \cdot 29243 \cdot 235531684534847 \\
a_4(1, 25)&=& 2^2 \cdot 3^2 \cdot 5^7 \cdot 11491 \cdot 102481 \cdot
4609259 \cdot 4679867
\end{array}\]

While Theorem 3 in~\cite{Duke-Jenkins} predicts that if $p^r \| n$,
then $p^{3r} | a_4(1, n)$, we see that in fact there are extra
powers of $2$ and $5$ dividing these coefficients.  For instance, we
expect to see that $2^3 | a_4(1, 2)$ and $5^3 | a_4(1, 5)$, while in
fact the stronger divisibility results $2^{10} | a_4(1, 2)$ and $5^4
| a_4(1, 5)$ are actually true.  Further computation reveals similar
divisibility by small primes for other modular forms of these
weights.

In this paper, we will prove the following theorem making these
divisibility results more explicit.  For an integer $N$, let $v_p(N)$ be the
$p$-adic valuation of $N$, or the largest integer $s$ such that $p^s
| N$.
\begin{theorem}\label{main} Let $k\in\{4,6,8,10,14\}$ and let $p\in\{2,3,5\}$.
Then for
$$\epsilon_{k,p}=\begin{cases}
7&\text{if $p=2$, $k=4,6,14$,}\cr
8&\text{if $p=2$, $k=8,10$,}\cr
2&\text{if $p=3$, $k=4,10$,}\cr
3&\text{if $p=3$, $k=6,8,14$,}\cr
1&\text{if $p=5$,}
\end{cases}$$
we have for all $m,n>0$,

$$v_p(a_k(m,n))\geq\begin{cases}
\epsilon_{k,p}&\text{if $v_p(m)>v_p(n)$},\cr
(v_p(n)-v_p(m))(k-1)+\epsilon_{k,p}&\text{if $v_p(n)>v_p(m)$.}\cr
\end{cases}$$

\end{theorem}
We note that the theorem makes no prediction about divisibility if
$v_p(m)=v_p(n)$.  If $(m, 30) = 1$, the theorem implies that for $a,
b, c > 0$, \[ a_k(m, 2^a 3^b 5^c n) \equiv 0 \pmod{2^{(k-1)a + 7}
3^{(k-1)b + 2} 5^{(k-1)c + 1}}.\]

Because any weakly holomorphic modular form of weight $k \in
\{4,6,8,10,14\}$ is a linear combination of the $f_{k,m}(\tau)$ and
the Eisenstein series $E_k(\tau)$, we can use this result to derive
divisibility results for many coefficients of more general weakly
holomorphic modular forms.  For instance, if
\[f(\tau)=c_{-2}q^{-2}+c_{-1}q^{-1}+\sum_{n>0}
c_nq^n=c_{-2}f_{4,2}(\tau)+c_{-1}f_{4,1}(\tau)\] is any modular form
of weight $4$ with integer coefficients, constant coefficient 0, and
a pole of order $2$ at the cusp, we can deduce easily that for
$3|n$, $v_3(c_n)\geq 3v_3(n)+2$.

\section{Definitions}
In this section, we define various modular forms and operators that will be used
throughout the paper.
\subsection{Modular forms derived from Eisenstein series}

We define the standard normalized Eisenstein series of even weight
$k\geq 4$ as
$$E_k(\tau)=1+A_k\sum_{n=1}^\infty\sigma_{k-1}(n)q^n,$$
where $A_k=-2k/B_k$ and $B_k$ is the $k$th Bernoulli number.
In addition, we will need the forms
$$S_{k,p}(\tau)=\frac{E_k(\tau)-E_k(p\tau)}{A_k},\qquad \text{and}\qquad T_{k,p}(\tau)=
\frac{p^kE_k(p\tau)-E_k(\tau)}{p^k-1}.$$
For $k\in\{4,6,8,10,14\}$ and any $p$, it is well known that $A_k$
is an integer and it is  easily seen that $S_{k,p}$ has integral
Fourier coefficients.  For $k\in\{4,6\}$ and $p=2$ and for $k=4$ and
$p=3$, one checks easily that $T_{k,p}$ has integer coefficients.
Often, when $p$ is clear from context, we will write $S_k$ and $T_k$
in place of $S_{k,p}$ and $T_{k,p}$.  Note that $S_{k,p}$ and
$T_{k,p}$ are modular forms of weight $k$ and level $p$, with
$S_{k,p}$ vanishing at $\infty$ and $T_{k,p}$ vanishing at 0.

In weight 2, we will need the Eisenstein series
$E_2(\tau)=1-24\sum_{n=1}^\infty\sigma_1(n)q^n$, which,  although
not modular itself, is used to produce the modular form
$$\frac{E_2(\tau)-pE_2(p\tau)}{1-p}$$ of weight 2 and level $p$ \cite[p.
88]{Stein-comp}.

\subsection{Newforms}

In level 2, we use the following notation for the normalized newforms of weight $8$ and $10$:
$$\Xi_8(\tau)=(\eta(\tau)\eta(2\tau))^8\qquad\text{and}\qquad \Xi_{10}(\tau)=S_{4,2}(\tau)T_{6,2}(\tau).$$
In level 3, we use the following notation for the normalized newform
of weight $6$:
$$\Omega_6(\tau)=(\eta(\tau)\eta(3\tau))^6.$$
In level 5, we use the following notation for the normalized
newforms of weights $4$ and $6$:
$$\Lambda_4(\tau)=(\eta(\tau)\eta(5\tau))^4\qquad\text{and}\qquad\Lambda_6(\tau)=\frac{5E_2(5\tau)-E_2(\tau)}4\Lambda_4(\tau).$$

These newforms can all be found in William Stein's online tables of
modular forms~\cite{Stein},  and we check there that the sign of the
Fricke involution is $1$ for $\Xi_8$ and $\Lambda_4$ and $-1$ for
$\Xi_{10}$, $\Omega_6$, and $\Lambda_6$.

\subsection{Weakly holomorphic forms of weight 0}

Following Apostol \cite[pg. 87]{Apostol}, for $p\in\{2,3,5\}$ we define $\lambda=\lambda_p=24/(p-1)$,
$$\Phi(\tau)=\Phi_p(\tau)=\left(\frac{\eta(p\tau)}{\eta(\tau)}\right)^\lambda, $$
and $\psi(\tau)=1/\Phi(\tau)$.  Although $\lambda$, $\Phi$, and
$\psi$ depend on $p$, we often omit this  dependence from the
notation, since it will be clear from context.

We recall from \cite{Apostol} that both $\psi$ and $\Phi$ are in $M_0^!(p)$, both have integer Fourier coefficients,
$\Phi$ has a zero at $\infty$ and a pole at 0, $\psi$ has a pole at
$\infty$ and a zero at 0, and
$$\psi\left(\frac{-1}{p\tau}\right)=p^{\lambda/2}\Phi(\tau).$$

\subsection{Operators on modular forms}

For $p$ a prime, and $f(\tau)=\sum_{n\geq n_0}a_nq^n$ a weakly
holomorphic modular  form, we define the $U_p$ operator by
$$(f|U_p)(\tau)=\sum_{n\geq n_0/p}a_{pn}q^n.$$
An alternative definition of $f|U_p$ that we will find useful is \cite[Thm. 4.5]{Apostol}
$$(f|U_p)(\tau)=\frac{1}{p}\sum_{j=0}^{p-1}f\left(\frac{\tau+j}p\right).$$
If $f$ has weight $k$ and level $N$, then $f|U_p$ is again a weakly
holomorphic modular  form of weight $k$ and level $N$ (if $p|N$) or
level $pN$ (if $p\nmid N$) \cite[Prop. 2.22]{Ono}.

We also define the operators $V_p$ by $(f|V_p)(\tau)=f(p\tau)$ and
the Hecke  operators $T_p$ by $(f|T_p)=(f|U_p)+p^{k-1}(f|V_p)$.
Note that $V_p$ takes modular forms of weight $k$ and level $N$ to
forms of weight $k$ and level $Np$, while $T_p$ preserves both
weight and level \cite[Prop. 2.2, Thm.  4.5]{Ono}.

\section{Reduction to the case where $p|m$, $(n,p)=1$}
In this section, we will prove that if Theorem~\ref{main} is true
for the special case of Fourier coefficients $a_k(m,n)$ with $p|m$
and $(n,p)=1$, then it is true in all cases. We begin by citing a
result of Duke and Jenkins relating certain Fourier coefficients,
from which the divisibility result of \cite{Duke-Jenkins} cited in
the introduction follows immediately.

\begin{proposition}[\cite{Duke-Jenkins}, Lemma 1] Let $p$ be a prime and $k\in\{4,6,8,10,14\}$.  Then for
$m,n,s\in\Z$, with $m,n,s>0$ and $p$ a prime,
$$a_k(m,np^s)=p^{s(k-1)}\left(a_k(mp^s,n)-a_k(mp^{s-1},n/p)\right)+a_k(m/p,np^{s-1}).$$
\end{proposition}
\noindent Applying induction to this proposition, we obtain the
following.
\begin{corollary}\label{induction}
Let $(m,p)=(n,p)=1$, $r,s\geq 0$, $k\in\{4,6,8,10,14\}$.  Then for $0\leq t\leq \min(r,s-1)$,
$$a_k(mp^r,np^s)=a_k(mp^{r-t-1},np^{s-t-1}) + \sum_{j=0}^tp^{(s-j)(k-1)}a_k(mp^{r+s-2j},n).$$
\end{corollary}
\noindent Applying this corollary, we obtain the following
reduction.
\begin{theorem} Let $p\in\{2,3,5\}$ and let $k\in\{4,6,8,10,12\}$.  Assume that
$v_p(a_k(a,b))\geq \epsilon_{k,p}$ for all $a,b>0$ having $p|a$ and
$p\nmid b$. Let $m,n>0$ be relatively prime to $p$ and $r,s\geq 0$.  Then if $r>s$,
$$v_p(a_k(mp^r,np^s))\geq \epsilon_{k,p},$$
and if $r<s$,
$$v_p(a_k(mp^r,np^s))\geq (s-r)(k-1)+\epsilon_{k,p}.$$
\end{theorem}
\begin{proof} For $r>s$, we apply Corollary~\ref{induction} with $t=s-1$ to $a_k(mp^r,np^s)$,
and note that each term obtained is divisible by $p^{\epsilon_{k,p}}$.

For $r<s$, we apply Corollary~\ref{induction} with $t=r$.  In this
case, we note that the term  outside the sum vanishes, and each term
inside the sum is divisible by $p^{(s-r)(k-1)+\epsilon_{k,p}}$.
\end{proof}

Hence, we see that in order to prove Theorem~\ref{main}, we need only prove it for $a_k(m,n)$ when $p|m$ and $p\nmid n$.

\section{Poles at zero of weakly holomorphic modular forms}

In this section, we derive a formula which will allow us to compute
the Fourier expansion at 0 of $f|U_p$, where $f$ is a weakly
holomorphic modular form for which we know the Fourier expansion at
$\infty$.  This will allow us to use the expansion at 0 of $f|U_p$
to obtain information about the Fourier coefficients in the
expansion of $f$ at $\infty$.

\begin{lemma}\label{poles}
Let $f$ be a meromorphic modular form of weight $k$ on $SL_2(\Z)$,
and let $f_p=f|U_p$.  Then
$$p\tau^{-k}f_p(-1/\tau)=-f(\tau/p)+pf_p(\tau)+p^kf(p\tau).$$
\end{lemma}
Note the similarity of the lemma to Theorem 4.6 of \cite{Apostol}.
\begin{proof} For an integer $j$ with $1\leq j\leq p-1$, we will denote by $j'$ the unique
integer with $-(p-1)\leq j\leq -1$ such that $jj'\equiv 1\pmod p$, and we will write $b_j=(jj'-1)/p$.  We note that
\begin{align*}
p\tau^{-k}f_p(-1/\tau)&=p\tau^{-k}\frac1p\sum_{j=0}^{p-1}f\left(\frac{-1/\tau+j}{p}\right)\cr
&=\tau^{-k}\sum_{j=0}^{p-1}f\left(\frac{j\tau-1}{p\tau}\right)\cr
&=\tau^{-k}\sum_{j=1}^{p-1}f\left(\begin{pmatrix}j&b\cr
p&j'\end{pmatrix}
\begin{pmatrix}1&-j'\cr0&p\end{pmatrix}\tau\right)+\tau^{-k}f(-1/(p\tau))\cr
&=\tau^{-k}\sum_{j=1}^{p-1}\left(p\left(\frac{\tau-j'}{p}\right)+j'\right)^k
f\left(\frac{\tau-j'}{p}\right)+\tau^{-k}(p\tau)^kf(p\tau))\cr
&=-f(\tau/p)+\sum_{j=0}^\infty
f\left(\frac{\tau+j}{p}\right)+p^kf(p\tau)\cr
&=-f(\tau/p)+pf_p(\tau)+p^kf(p\tau).
\end{align*}
\end{proof}

\begin{corollary} \label{swap}
Let $f$ be a meromorphic modular form of weight $k$ on $\Gamma$, and
let $f_p=f|U_p$.  Then
$$p(p\tau)^{-k}f_p(-1/pt)=-f(\tau)+pf_p(p\tau)+p^kf(p^2\tau).$$
Further, $p(p\tau)^{-k}f_p(-1/pt)$ is modular of weight $k$ and level $p$.
\end{corollary}

\begin{proof}
The equality follows immediately by replacing $\tau$ by $p\tau$ in
Lemma~\ref{poles}.  The statement about the weight and level follows
from the fact that $pf_p(p\tau)+p^kf(p^2\tau)=p((f|T_p)|V_p)$, along with
the fact that $T_p$ preserves weight and level, while $V_p$ raises
levels by a factor of $p$.
\end{proof}

\section{Integral bases for spaces of modular forms}
As seen in the previous section, modular forms on $\Gamma_0(p)$ can
be used to study associated modular forms on $\SL_2(\Z)$. Therefore,
for $p=2, 3, 5$ we now construct integral bases for the spaces
$M_k(p)$, or bases having the property that the modular forms in
$M_k(p)$ with integer coefficients are exactly the integer linear combinations
of the basis elements. These bases allow us to study divisibility
properties of Fourier coefficients by studying the first several
coefficients of a given form.

\begin{lemma}\label{integralbasis} Let $p\in\{2,3,5\}$, let $k \geq 0$ be even, and
let $d=\dim M_k(p)$.  Then there is a basis $\{B_{n,k,p}:0\leq
n<d\}$ of $M_k(p)$ such that each $B_{n,k,p}=q^n+{O}(q^d)$ and each
$B_{n,k,p}$ has integer coefficients.
\end{lemma}
\begin{proof}
To construct the basis, for each weight $k$ we will find a modular
form $f$ in $M_k(p)$ with integer coefficients and leading
coefficient $1$ that vanishes with order $d- 1$ at $\infty$.  We can
then multiply $f$ by $\psi(\tau)$ and subtract off an appropriate
integer multiple of $f$ to get a form with $q$-expansion beginning
$q^{d - 2} + {O}(q^d)$. Repeating this process of multiplying by
$\psi$ and subtracting earlier basis elements, we generate a basis
for $M_{k}(p)$ of modular forms with integer coefficients and the
desired Fourier expansions.

To construct these modular forms vanishing to order $d-1$, we note
that standard dimension formulas \cite[Prop. 6.1]{Stein-comp} yield
the following values for $d=\dim(M_k(p)$:
$$\dim (M_{k}(2)) = \left\lfloor \frac{k}{4} \right\rfloor + 1,\quad
\dim (M_{k}(3)) = \left\lfloor \frac{k}{3} \right\rfloor + 1,\quad\text{ and }\quad
\dim (M_{k}(5)) = 2 \left\lfloor \frac{k}{4} \right\rfloor + 1.$$

In level 2, we find that for weight $0$, the constant function $1$
is the desired modular form, and for weight $2$, the form
$2E_2(2\tau) - E_2(\tau)$ works.  We then note that for $k \geq 0$,
$\textrm{dim}(M_{k+4}(2)) =\textrm{dim}(M_k(2))+ 1$, so the
appropriate form of weight $k+4$ can be obtained from the form of
weight $k$ by multiplying by $S_{4, 2}(\tau)$, which is of weight
$4$ and has a Fourier expansion at $\infty$ beginning $q +
{O}(q^2)$.

For level $3$, in weights $0, 2, 4$ we have $1, \frac12(3E_2(3\tau)
- E_2(\tau)), S_{4, 3}$ respectively. Additionally,
$\textrm{dim}(M_{k+6}(3)) =\textrm{dim}(M_k(3))+ 2$, so multiplying
by a form of weight $6$ with Fourier expansion beginning with $q^2$
suffices to construct the needed forms of higher weight; the form
\[\frac{\eta^{18}(3\tau)}{\eta^{6}(\tau)}=\Phi_3(\tau)\Omega_6(\tau)
= q^2 + 6 q^3 + \cdots\] works.

For level $5$, $\textrm{dim}(M_{k+4}(5)) =\textrm{dim}(M_k(5))+ 2$,
and the weight $4$ form $\eta^{10}(5\tau) \eta^{-2}(\tau)=\Phi_5(\tau)\Lambda_4(\tau)$, along
with the constant function $1$ and the weight $2$ form $\frac14(5E_2(5\tau)
- E_2(\tau))$, suffice to construct the basis in all weights.
\end{proof}

Note that the set $\{B_{n,k,p}\}$ does in fact form an integral
basis of $M_k(p)$, since any form $$f(\tau)=\sum_{n=0}^{\infty}
a_nq^n\in M_k(p)$$ with integer coefficients can clearly be written
as the integer linear combination
$$f(\tau)=\sum_{n=0}^{d-1}a_nB_{k,n,p}(\tau).$$

The importance of this basis for our purposes is that it allows us
to check the divisibility of all coefficients of a modular form by
checking only finitely many coefficients.

\begin{lemma}\label{congruence-basis}
Let $k\geq 0$ be even  and let $p\in\{2,3,5\}$.  Let $d=\dim M_k(p)$.  If $F(\tau)\in M_k(p)$ has integer
coefficients and its first $d$ coefficients are divisible by $p^s$,
then $F(\tau)\equiv 0\pmod {p^s}$.
\end{lemma}
\begin{proof}
For $F(\tau)=\sum_{n=0}^\infty a_nq^n$, we have
$F(\tau)=\sum_{n=0}^{d-1} a_nB_{k,n,p}(\tau)$.  If each $a_n$ with
$0\leq n\leq d-1$ is divisible by $p^s$, then clearly every
coefficient of $F(\tau)$ is divisible by $p^s$.
\end{proof}

\section{Weakly holomorphic modular forms of negative weight with minimal poles}
Because of the duality of coefficients of basis elements of the
weights $k$ and $2-k$ (Corollary~\ref{duality}), for a fixed $n$ we
can study the coefficients $a_k(mp^s, n)$ of the forms $f_{k, mp^s}$
for all positive values of $m$ and $s$ simply by studying the Fourier
coefficients of the single negative weight form $f_{2-k, n}|U_p$.
To facilitate our study of these forms, we now examine negative weight weakly holomorphic modular
forms of level $p$ which are holomorphic except for a pole at
exactly one of 0 and $\infty$. Since there are no holomorphic
modular forms of negative weight, we see that in fact such a form
must have a pole of order at least one at one of 0 and $\infty$.  In
fact, using a valence formula, we will see that in many cases, such
a modular form must have a pole of order greater than one at one of
the two cusps.

We begin by bounding the order of the pole.
\begin{proposition} Let $p\in\{2,3,5\}$, and let $k<0$ be even.  Let $f\in M_k^!(p)$ and suppose that $f$ is
holomorphic at 0.  Denote by  ${\rm ord}_\infty(f)$ the order of
vanishing of $f$ at $\infty$ (so if $f$ has a pole, this is
negative).  Then
$${\rm ord}_\infty(f)\leq k\left(\frac{v_2(p)}{4}+\frac{v_3(p)}{3}\right),$$
where $v_2$ and $v_3$ are given by \cite[p. 535]{El-Guindy} (see
also \cite[p. 25]{Shimura}).   If equality holds, then $f$ must be
nonvanishing in the upper half plane and at $0$.
\end{proposition}

We note that $v_2(2)=1$, $v_2(3)=0$ and $v_2(5)=2$, while
$v_3(2)=0$, $v_3(3)=1$ and $v_3(5)=0$.   In addition, we note that
reversing the roles of 0 and $\infty$ yields an analogous bound
for modular functions holomorphic at $\infty$ with a pole at 0.

\begin{proof} This follows immediately from the valence formula \cite[(3.9)]{El-Guindy}, by
noticing that for a weakly holomorphic modular form whose only pole
is at $\infty$, all the  orders of vanishing at any point except
$\infty$ must be nonnegative. For equality, it is clear that the
order of vanishing at any non-infinite point must be 0.
\end{proof}

For $k<0$ and even and $p$ prime, we now define $\theta_{k,p}$ to be
a weakly holomorphic modular form of weight $k$  and level $p$ that
is holomorphic on the upper half plane and at $0$, with a pole of
minimal possible order at $\infty$ and leading coefficient 1.  We
define $\alpha_{k,p}$ in a similar way, but require that it be
holomorphic at $\infty$ and with minimal possible pole at 0.

Note that $\theta_{k,p}$ and $\alpha_{k,p}$ are well-defined, since
given two such objects,  their difference must be a holomorphic
modular form of negative weight, hence 0.  In addition, for the $k$
and $p$ that we study, $(p\tau)^{-k}\theta_{k,p}(-1/p\tau)$ can
easily be shown (using properties of the Fricke involution) to be a
multiple of $\alpha_{k,p}$.  It will be convenient to write this
relation as
$\theta_{k,p}(-1/p\tau)=\mu_{k,p}\tau^k\alpha_{k,p}(\tau)$.  We also
note that often, $\alpha_{k,p}\equiv 1\pmod {p^{\nu_{k,p}}}$ for
some value of $\nu$.  In Tables~\ref{theta-alpha-2},
~\ref{theta-alpha-3}, and \ref{theta-alpha-5} we define
$\theta_{k,p}$ and $\alpha_{k,p}$ for $k\in\{-2,-4,-6,-8,-12\}$ and
$p\in\{2,3,5\}$, and give the values of $\mu_{k,p}$ and $\nu_{k,p}$.

\begin{table}[t]
\begin{minipage}{2.3in}
\begin{tabular}{| c | c | c | c | c |}
\hline $k$&$\theta_{k,2}$&$\alpha_{k,2}$&$\mu_{k,2}$&$\nu_{k,2}$\cr
\hline
$-2$&$\displaystyle{\frac{\Xi_{10}(\tau)}{\Delta(2\tau)}}$&$\displaystyle{\frac{\Xi_{10}(\tau)}{\Delta(\tau)}}$&$-2^5$&$3$\cr
\hline
$-4$&$\displaystyle{\frac{\Xi_8(\tau)}{\Delta(2\tau)}}$&$\displaystyle{\frac{\Xi_8(\tau)}{\Delta(\tau)}}$&$2^4$&$4$\cr
\hline
$-6$&$\displaystyle{\frac{T_6(\tau)}{\Delta(2\tau)}}$&$\displaystyle{\frac{S_6(\tau)}{\Delta(\tau)}}$&$-2^9$&$3$\cr
\hline
$-8$&$\displaystyle{\frac{T_4(\tau)}{\Delta(2\tau)}}$&$\displaystyle{\frac{S_4(\tau)}{\Delta(\tau)}}$&$2^8$&$5$\cr
\hline
$-12$&$\displaystyle{\frac{\Delta(\tau)}{\Delta^2(2\tau)}}$&$\displaystyle{\frac{\Delta(2\tau)}{\Delta^2(\tau)}}$
&$2^{12}$&$4$\cr \hline
\end{tabular}
\caption{Values of $\theta_{k,p}$, $\alpha_{k,p}$, $\mu_{k,p}$ and $\nu_{k,p}$ for $p=2$.}\label{theta-alpha-2}
\end{minipage}
\hfill
\begin{minipage}{3in}
\begin{tabular}{| c | c | c | c | c |}
\hline $k$&$\theta_{k,3}$&$\alpha_{k,3}$&$\mu_{k,3}$&$\nu_{k,3}$\cr
\hline
$-2$&$\displaystyle{\frac{S_4(\tau)\Omega_6(\tau)}{\Delta(3\tau)}}$&$\displaystyle{\frac{T_4(\tau)\Omega_6(\tau)}{\Delta(\tau)}}$
&$-3^2$&$1$\cr \hline
$-4$&$\displaystyle{\frac{S_4(\tau)T_4(\tau)}{\Delta(3\tau)}}$&$\displaystyle{\frac{S_4(\tau)T_4(\tau)}{\Delta(\tau)}}$
&$3^4$&$1$\cr \hline
$-6$&$\displaystyle{\frac{\Omega_6(\tau)}{\Delta(3\tau)}}$&$\displaystyle{\frac{\Omega_6(\tau)}{\Delta(\tau)}}$&$-3^3$&$2$\cr
\hline
$-8$&$\displaystyle{\frac{T_4(\tau)}{\Delta(3\tau)}}$&$\displaystyle{\frac{S_4(\tau)}{\Delta(\tau)}}$&$3^5$&$1$\cr
\hline
$-12$&$\displaystyle{\frac{\Phi(\tau)\Delta(\tau)}{\Delta^2(3\tau)}}$
&$\displaystyle{\frac{\psi(\tau)\Delta(3\tau)}{\Delta^2(\tau)}}$&$3^6$&$2$\cr
\hline
\end{tabular}
\caption{Values of $\theta_{k,p}$, $\alpha_{k,p}$, $\mu_{k,p}$ and $\nu_{k,p}$ for $p=3$.}\label{theta-alpha-3}
\end{minipage}
\end{table}

\begin{table}
\begin{center}
\begin{tabular}{| c | c | c | c | c |}
\hline $k$&$\theta_{k,5}$&$\alpha_{k,5}$&$\mu_{k,5}$&$\nu_{k,5}$\cr
\hline
$-2$&$\displaystyle{\frac{\Lambda_4(\tau)\Lambda_6(\tau)\Phi(\tau)}{\Delta(5\tau)}}$
&$\displaystyle{\frac{\Lambda_4(\tau)\Lambda_6(\tau)\psi(\tau)}{\Delta(\tau)}}$&$-5^2$&$0$\cr
\hline
$-4$&$\displaystyle{\frac{\Lambda_4(\tau)^2\Phi(\tau)}{\Delta(5\tau)}}$
&$\displaystyle{\frac{\Lambda_4^2(\tau)\psi(\tau)}{\Delta(\tau)}}$&$5$&$1$\cr
\hline $-6$&$\displaystyle{\frac{\Lambda_6(\tau)}{\Delta(5\tau)}}$
&$\displaystyle{\frac{\Lambda_6(\tau)}{\Delta(\tau)}}$&$-5^3$&$0$\cr
\hline $-8$&$\displaystyle{\frac{\Lambda_4(\tau)}{\Delta(5\tau)}}$
&$\displaystyle{\frac{\Lambda_4(\tau)}{\Delta(\tau)}}$&$5^2$&$1$\cr
\hline
$-12$&$\displaystyle{\frac{\Phi(\tau)\Lambda_4^3(\tau)}{\Delta^2(5\tau)}}$
&$\displaystyle{\frac{\psi(\tau)\Lambda_4^3(\tau)}{\Delta^2(\tau)}}$&$5^3$&$1$\cr
\hline
\end{tabular}
\end{center}
\caption{Values of $\theta_{k,p}$, $\alpha_{k,p}$, $\mu_{k,p}$ and $\nu_{k,p}$ for $p=5$.}\label{theta-alpha-5}
\end{table}

We remark that in all cases except $p=5$, $k=-2$ or $-6$, the
minimum order  of the pole predicted by the valence
formula was achieved.  For $p=5$, $k=-2$, the valence formula
predicts that $\theta_{k,p}$ should have a pole of order at least 1
(since $k(v_2(5)/4+v_3(5)/3)=-1$).  If  equality were achieved,
then $\theta_{k,p}$ would have to be nonvanishing on the upper half
plane and at 0, so that $1/\theta_{k,p}$ would be a holomorphic
modular form of weight 2 and level 5, vanishing at $\infty$.  Since
no such form exists, we see that $\theta_{-2,5}$ must have at least
a double pole.  Similarly, for $p=5$ and $k=-6$ the bound arising
from the valence formula indicates that $\theta_{-6,5}$ must have a
pole of order at least 3.  However, if $\theta_{-6,5}$ had a pole of
order 3, then its reciprocal would be a nonzero holomorphic modular
form of weight 6, level 5 that vanishes at $\infty$ with order 3.
The basis computations in the previous section show that such an
object cannot exist.  Hence, the order of the pole of
$\theta_{-6,5}$ must be at least 4.

The congruences on $\alpha_{k,p}$ follow from the following lemma.

\begin{lemma} \label{alpha}
Let $\alpha(\tau)=\sum_{n=0}^\infty a_nq^n$ be a weight $k$ weakly
holomorphic  modular form of level p, with $k$ even and negative.
Suppose that $\alpha$ has integer coefficients, is holomorphic at
$\infty$ with $a_0=1$, and has a pole of order $m$ at 0.  If
$p^r|a_i$ for all $1\leq i\leq m$ and there is some $F(\tau)\in
M_{-k}(p)$ having leading coefficient 1 and integer coefficients with $F(\tau)\equiv
1\pmod {p^r}$, then $\alpha(\tau)\equiv 1\pmod {p^r}$.
\end{lemma}

\begin{proof}
Note that $\alpha(\tau)F(\tau)-1=\sum_{n=1}^\infty b_nq^n$ is weakly
holomorphic  of weight 0 and is holomorphic at $\infty$, with a pole
of order at most $m$ at 0 and vanishing at $\infty$.  As such, it
must be an integer linear combination of $\Phi(\tau),\ldots,\Phi^m(\tau)$.
Since $F(\tau)\equiv 1\pmod{ p^r}$, we see that each $b_n\equiv
a_n$, so that each $b_n$ is divisible by $p^r$.  If we write
$$\alpha(\tau)F(\tau)-1=\sum_{n=1}^\infty b_nq^n=\sum_{n=1}^m c_n\Phi^n(\tau)$$
it is clear that $c_1=b_1$, and that $c_2\equiv b_2\pmod {b_1}$, so
that $c_2\equiv b_2\equiv 0\pmod{p^r}$.  Inductively, we find that
each $c_n\equiv 0\pmod {p^r}$, so that $\alpha(\tau)F(\tau)-1\equiv
0\pmod{p^r}$.  Since $F(\tau)\equiv 1\pmod {p^r}$, we see that
$\alpha(\tau)\equiv 1\pmod{ p^r}$.
\end{proof}

We apply this lemma to each of the $\alpha_{k,p}$ with
$k\in\{-2,-4,-6,-8,-12\}$.   For $k=-4,-6,-8$, we use the Eisenstein
series of weight $-k$ for $F(\tau)$.  For $k=-12$, we set $F(\tau)=E_4(\tau)^3$.  For $k=-2$ and $p=2,3$,  we
let $F(\tau)$ be the unique monic form of weight 2, level $p$, which
is given by \cite[p. 88]{Stein-comp}
$$\frac{E_2(\tau)-pE_2(p\tau)}{1-p}$$
and is congruent to 1 mod {8} in level 2, and congruent to 1 mod 3
in level 3.  In each case, one checks easily that $F(\tau)\equiv
1\pmod {p^{\nu_{k,p}}}$, and that the first $m$ coefficients of
$\alpha_{k,p}$ are divisible by $p^{\nu_{k,p}}$. Note that there is
no congruence on the coefficients of $\alpha_{-2,5}$ modulo $5$.

\section{Reduction to finitely many coefficients}
To prove Theorem~\ref{main}, it remains to show that for any integer
$s \geq 1$, the coefficient $a_k(mp^s, n)$ of the form $f_{k, mp^s}$
is sufficiently divisible by $p$ when $p \nmid n$.  To accomplish
this, we define
$$g(\tau)=f_{k,mp^s}(\tau)-f_{k,mp^{s-1}}(p\tau)=\sum_{n=1}^\infty
\gamma_nq^n.$$ Note that for $p\nmid n$, $\gamma_n=a_k(mp^s,n)$.  We
will prove that all of the $\gamma_n$ are divisible by
$p^{\epsilon_{k,p}}$. In this section, we show that it is actually
sufficient to prove that finitely many of the $\gamma_n$ satisfy the
desired divisibility.

\begin{lemma} \label{congruence}
Let $p\in\{2,3,5,7,13\}$, $k\in\{4,6,8,10,14\}$,
$\lambda=24/(p-1)$, and $m>0$.  Let  $g(\tau)=f_{k,mp^s}-f_{k,mp^{s-1}}(p\tau)$
with $s>0$.  Then we have that $$g(\tau)\equiv
F(\tau)\pmod{p^{\lambda/2}}$$ for some holomorphic modular form $F\in M_k(p)$, where $F$ has integer coefficients and
vanishes at $\infty$.
\end{lemma}

\begin{proof}
We immediately notice that $g(\tau)\in M_k^!(p)$ and vanishes at
$\infty$.  We are then interested in its behavior at the cusp $0$ of
$\Gamma_0(p)$.

We have $g(t)= f_{k,mp^s}(\tau)-f_{k,mp^{s-1}}(p\tau)$.
Hence,
\begin{align*}
\tau^{-k}g(-1/\tau)&=\tau^{-k}\left(f_{k,mp^s}(-1/\tau)-f_{k,mp^{s-1}}(-p/\tau)\right)\cr
&=\tau^{-k}\left(\tau^kf_{k,mp^s}(\tau)-\left(\frac{\tau}p\right)^kf_{k,mp^{s-1}}(\tau/p)\right)\cr
&=f_{k,mp^s}(\tau)-\frac1{p^k}f_{k,mp^{s-1}}(\tau/p).
\end{align*}
Replacing $\tau$ by $p\tau$, we obtain
\begin{align*}(p\tau)^{-k}g(-1/(p\tau))&=f_{k,mp^s}(p\tau)-\frac1{p^k}f_{k,mp^{s-1}}(\tau).\cr
&=q^{-mp^{s+1}}-\frac1{p^k}q^{mp^{s-1}}+{O}(q)\cr
&=\sum_{i=1}^{mp^{s+1}}B_i\psi(\tau)^iE_k(\tau)+h(\tau)
\end{align*}
where  $p^kB_i\in\Z$ for all $i$, and where $h(\tau)$ is modular on $\Gamma_0(p)$ and holomorphic at both $0$ and $\infty$.

We now replace $\tau$ by $-1/p\tau$, to obtain
\begin{align*}
\tau^kg(\tau)&= \sum_{i=1}^{mp^{s+1}}B_i\psi(-1/p\tau)^iE_k(-1/p\tau)+h(-1/p\tau)\cr
&=\sum_{i=1}^{mp^{s+1}}B_i(p^{\lambda/2}\Phi(\tau))^i(p\tau)^kE_k(p\tau)+h(-1/p\tau)\cr
&=\sum_{i=1}^{mp^{s+1}}(p^kB_i)(p^{\lambda/2}\Phi(\tau))^i\tau^kE_k(p\tau)+h(-1/p\tau)\cr
\end{align*}
Dividing by $\tau^k$, writing $A_i=p^kB_i\in\Z$, and setting $F(\tau)=\tau^{-k}h(-1/p\tau)$, we see that
$$g(\tau)=p^{\lambda/2}\sum_{i=1}^{mp^{s+1}}A_i(p^{\lambda/2})^{i-1}\Phi(\tau)^iE_k(p\tau)+F(\tau).$$

We note that $F(\tau)$ has the following properties: it is modular
of weight $k$ on  $\Gamma_0(p)$ and has integer coefficients (since
both $g(\tau)$ and the sum have these properties), and it is
holomorphic at both cusps of $\Gamma_0(p)$ (since $h(\tau)$ has this
property).  Further, since both $g(\tau)$ and the sum vanish at
$\infty$, it is clear that $F(\tau)$ vanishes at $\infty$.
\end{proof}

Now, for a given $p\in\{2,3,5\}$, since  $\lambda/2>\epsilon_{k,p}$,
we can study the divisibility of the coefficients of $g(\tau)$ by
$p^{\epsilon_{k,p}}$ by studying the coefficients of $F(\tau)$.
Since $F(\tau)\in M_k(p)$, we can apply Lemma~\ref{congruence-basis}
to it to obtain the following corollary.

\begin{corollary}\label{gamma}
Let $k\in\{4,6,8,10,14\}$, $p\in\{2,3,5\}$, $m,s>0$.  Let $d$ be the
dimension of $M_k(p)$  If
$$g(\tau)=f_{k,mp^s}(\tau)-f_{k,mp^{s-1}}(p\tau)=\sum_{n=1}^\infty\gamma_nq^n,$$
and the $\gamma_n$ with $1\leq n\leq d-1$ are all divisible by
$p^{\epsilon_{k,p}}$, then all of the $\gamma_n$ are divisible by
$p^{\epsilon_{k,p}}$.  In particular, if $p\nmid n$, we have
$a_k(mp^s,n)=\gamma_n$ and thus $v_p(a_k(mp^s,n)\geq\epsilon_{k,p}$.
\end{corollary}
\begin{proof} By Lemma~\ref{congruence}, we have that $g(\tau)\equiv F(\tau)\pmod{p^{\lambda/2}}$,
for some $F(\tau)\in M_k(p)$ that vanishes at infinity.  Write
$$F(\tau)=\sum_{n=1}^\infty d_nq^n.$$  Since $\lambda/2>\epsilon_{k,p}$, the assumptions of
the corollary show that $p^{\epsilon_{k,p}}|d_n$ for all $1\leq
n<d$.  Lemma~\ref{congruence-basis}  then shows that
$p^{\epsilon_{k,p}}|d_n$ for all $n$, and hence that
$p^{\epsilon_{k,p}}|\gamma_n$ for all $n$, as desired.
\end{proof}

Corollary~\ref{gamma} allows us to test divisibility of all the
$\gamma_n$ (for a given $m$) by testing finitely many of them.  We
now prove a theorem which uses duality to  allow us to test
divisibility of all the $\gamma_n$ (for arbitrary $m>0$) by studying
a small number of modular forms.

\begin{theorem}\label{forms-to-test} Let $p\in\{2,3,5\}$, $k\in\{4,6,8,10,14\}$, and let  $d$ be the dimension of $M_k(p)$.  If
$$(f_{2-k,j}|U_p)(\tau)\equiv a_{2-k}(j,0)\pmod {p^{\epsilon_{k,p}}}$$
for each $j$ with $1\leq j\leq d-1$ and $p\nmid j$, and
$$(f_{2-k,j}|U_p)(\tau)-f_{2-k,j/p}(\tau)\equiv a_{2-k}(j,0)-a_{2-k}(j/p,0)\pmod{ p^{\epsilon_{k,p}}}$$
for each $j$ with $1\leq j\leq d-1$ and $p|j$, then
$p^{\epsilon_{k,p}}|a_k(mp^s,n)$ for all $s > 0$ and $m,n>0$ prime to
$p$.
\end{theorem}

\begin{proof}
By Corollary~\ref{gamma}, we need to show that for any $m,s>0$,
$p^{\epsilon_{k,p}}$ divides each $\gamma_n=a_k(mp^s,n)$ for $1\leq
n\leq d-1$  with $p\nmid n$ and each
$\gamma_n=a_k(mp^s,n)-a_k(mp^{s-1},n/p)$ for $1\leq n\leq d-1$ with
$p|n$.  Using duality (Corollary~\ref{duality}), we see that this is
equivalent to checking that for $1\leq n<d$,  $p^{\epsilon_{k,p}}$
divides each $a_{2-k}(n,mp^s)$ when $p\nmid n$ and each
$a_{2-k}(n,mp^s)-a_{2-k}(n/p,mp^{s-1})$ when $p|n$. As $m$ and $s$
run through all positive integers, the $a_{2-k}(n,mp^s)$ are just
the  nonconstant coefficients of $f_{2-k,n}|U_p$, and the
$a_{2-k}(n,mp^s)-a_{2-k}(n/p,mp^{s-1})$ are the nonconstant
coefficients of $(f_{2-k,n}|U_p)-f_{2-k,n/p}$.
\end{proof}

To conclude the proof of Theorem~\ref{main}, we must verify that the
modular forms in Theorem~\ref{forms-to-test} satisfy the congruences
described there.  In the next section, we describe how to prove
these congruences by computing finitely many coefficients of each
form, and we summarize the results of these computations.

\section{Conclusion}
In examining the case $p=2$ (as well as some weights for $p=3$ or
$p=5$), we use the following lemma to reduce the congruences we
desire to the computation of a single coefficient of a modular form.
Since the modular forms in question are easily computable,
Theorem~\ref{main} follows.

\begin{lemma} Let $p\in\{2,3,5\}$ and $k\in\{4,6,8,10,14\}$.  Let $f(\tau)\in M_{2-k}^!(p)$
be holomorphic at $\infty$. Suppose that $f(\tau)$ and
$p^{k-2}\left(p(p\tau)^{k-2}f(-1/p\tau)\right)$ have integer Fourier coefficients,
and that $v_p(\mu_{2-k,p})+1-k+\lambda_p/2\geq\epsilon_{k,2}$.  Let
$K$ be the constant Fourier coefficient of $f(\tau)$.  If
$v_p(K)\geq \epsilon_{k,p}-\nu_{2-k,p}$, then $f(\tau)\equiv
K\pmod{p^{\epsilon_{k,p}}}$.
\end{lemma}

\begin{proof}
The assumptions of the lemma indicate that
$(p\tau)^{k-2}f(-1/p\tau)$ has its only pole at $\infty$; hence, it
may be written as a linear combination
$$(p\tau)^{k-2}f(-1/p\tau)=p^{1-k}\sum_{i=0}^\infty A_i\psi^i(\tau)\theta_{2-k}(\tau),$$
where the $A_i\in\Z$.  Replacing $\tau$ by $-1/p\tau$, and dividing by $\tau^{2-k}$, we find that
$$f(\tau)=\sum_{i=0}^\infty A_i\mu_{2-k,p}p^{1-k+i\lambda_p/2}\Phi^i(\tau)\alpha_{2-k}(\tau).$$
Since $v_p(\mu_{2-k,p})+1-k+\lambda_p/2\geq\epsilon_{k,p}$ we see
that for $i>0$, the coefficient of $\Phi^i\alpha_{2-k}$ is divisible
by $p^{\epsilon_{k,p}}$.  Further, comparing constant coefficients,
we see that $K=A_0\mu_{2-k,p}p^{1-k}$.  Hence, we have that
$$f(\tau)\equiv K\alpha_k(\tau)\pmod{p^{\epsilon_{k,p}}}.$$
Under the assumption that $v_p(K)\geq \epsilon_{k,p}-\nu_{2-k,p}$,
and using that $\alpha_{2-k}\equiv 1\pmod{p^{\nu_{2-k,p}}}$, we see
that
$$f(\tau)\equiv K\pmod {p^{\epsilon_{k,p}}},$$
as desired.
\end{proof}

We note that all of the modular forms considered in
Theorem~\ref{forms-to-test} satisfy the holomorphicity condition of
this theorem, and Corollary~\ref{swap} guarantees that they satisfy
the integrality conditions. The inequality
$v_p(\mu_{2-k,p})+1-k+\lambda_p/2\geq\epsilon_{k,p}$ holds for $p=2$
(all weights), for $p=3$ in weights $k=4,6,10$, and for $p=5$ in
weight $k=4$. Hence, by calculating a single coefficient of each
form, we prove the desired divisibility properties for these cases.
In the cases where the theorem applies, we have calculated  the
necessary constant coefficients and verified Theorem~\ref{main}.

As an example, for $p=2$, $k=14$, we compute the constant
coefficients of $f_{-12,1}|U_2$, $f_{-12,2}|U_2-f_{-12,1}$, and
$f_{-12,3}|U_2$ to be, respectively, $24$, $196608$, and $38263776$.
Since each of these is divisible by $2^3$ and $\nu_{-12,2}=4$, we see that each of the
modular forms is congruent to a constant modulo $2^7$, proving Theorem~\ref{main} for $p=2$ in weight $14$.

We remark that this lemma does not immediately work in general
for $p=3$ and $p=5$ because we do not always have the inequality
$v_p(\mu_{2-k,p})+1-k+\lambda_p/2\geq\epsilon_{k,p}$. One could
prove a similar theorem, involving checking the divisibility of additional
coefficients, but we find it simpler to perform explicit
computations, as described below.

\subsection{Computational proof of the remaining cases}

For $p=3$, $k=8, 14$ and for $p=5$, $k>4$ we now describe the
calculations needed to prove the congruences in Theorem~\ref{main}
and give a single illustrative example.

Let $f$ be one of the modular forms of weight $2-k$ described in
Theorem~\ref{forms-to-test}.  We wish to show that $f$ satisfies the
congruence stated in Theorem~\ref{forms-to-test}.  We begin by using
Lemma~\ref{poles} to compute the Fourier expansion of
$p(p\tau)^{k-2}f(-1/p\tau)$ to high precision (say to
${O}(q^{500})$).  This allows us to easily determine a linear
combination
$$p(p\tau)^{k-2}f(-1/p\tau)=\sum_{i=0}^N{A_i}{p^{2-k}}\psi^i(\tau)\theta_{2-k}(\tau),$$
with $A_i\in\Z$.  Replacing $\tau$ by $-1/p\tau$ and dividing by $p\tau^{2-k}$, we obtain a linear combination
$$f(\tau)=\sum_{i=0}^N  A_i\mu_{2-k,p}p^{\lambda i/2+1-k}\Phi^i(\tau)\alpha_k(\tau)
=\sum_{i=0}^N B_i\Phi^i(\tau)\alpha_k(\tau).$$
In each case, we find that for $i>0$, $v_p( B_i)\geq\epsilon_{k,p}$,
so that
$$f(\tau)\equiv B_0\alpha_k(\tau)\pmod {p^{\epsilon_{k,p}}}.$$
We then check that $v_p(B_0)\geq\epsilon_{k,p}-\nu_{2-k,p}$, thereby
proving that
$$f(\tau)\equiv B_0\pmod{p^{\epsilon_{k,p}}},$$
as desired.

As an example (the simplest) for $p=3$, $k=8$, we must show that
$f_{-6,1}|U_3$ is congruent to a constant modulo $3^3$.  Following
the procedure above, we find that
$$(f_{-6,1}|U_3)(\tau)=\sum_{i=0}^{7} B_i\Phi^i(\tau)\alpha_{-6}(\tau)$$
with
$$B_0=-480,\quad B_1=-10451430,\quad B_2=-8628476076,$$
$$\quad B_3=-1922380466418,\quad B_4=-177993370102248,\quad B_5=-7892493396961545,$$
$$\quad B_6= -166771816996665690,\quad B_7=-1350851717672992089.$$
One checks easily that each $B_i$ with $i>0$ is divisible by
$3^{\epsilon_{8,3}}=3^3$ (indeed, even by $3^5$) and that $3|B_0$.
Since $\nu_{-6,3}=2$, we see that $f_{-6,1}|U_3\equiv -480\pmod
{3^3}$.

Note that there is nothing difficult about the computations; they
are nothing more than basic arithmetic with $q$-series.  They are
somewhat daunting to write down; for instance, if $p=5$ and $k =
14$, proving that $f_{-12,6}$ is congruent to a constant modulo 5
involves working with a linear combination of 144 terms of the form
$\Phi^i\alpha_{-12,5}$, many of which have coefficients over 290
digits long.  All of the computations were done using
GP/PARI~\cite{Pari}, and the scripts used are available from the
authors upon request.

These computations have been performed for $p=3$, $k=8,10,14$ and
$p=5$, $k=4,6,8,10,14$, and complete the proof that
Theorem~\ref{main} is true.

\bibliographystyle{amsplain}

\end{document}